\documentclass[a4paper,12pt]{article}

\usepackage[T2A]{fontenc}
\usepackage[cp1251]{inputenc}
\usepackage{amsfonts,amsmath,amsthm,amscd,amssymb,latexsym,amsbsy,pb-diagram,cite,caption}
\usepackage[mathscr]{eucal}

\usepackage{ifpdf}
\ifpdf
\usepackage[pdftex]{graphicx}
\usepackage[pdftex,dvipsnames,usenames]{color}
\else
\usepackage{graphicx}
\usepackage[dvips,dvipsnames,usenames]{color}
\fi

\newtheorem{definition}{Definition}[section]
\newtheorem{theorem}[definition]{Theorem}
\newtheorem{lemma}[definition]{Lemma}
\newtheorem{proposition}[definition]{Proposition}
\newtheorem{corollary}[definition]{Corollary}

\theoremstyle{definition}
\newtheorem{remark}[definition]{Remark}
\newtheorem{example}[definition]{Example}

\numberwithin{equation}{section}

 \DeclareMathOperator{\card}{card}
\DeclareMathOperator{\pr}{pr} \DeclareMathOperator{\id}{id}
\DeclareMathOperator{\diam}{diam}

\usepackage{wrapfig}
\begin{document}

\begin{center}
{\Large\bf Metric products and continuation of isotone functions}
\end{center}

\bigskip
\begin{center}
{\bf O. Dovgoshey, E. Petrov and G. Kozub}
\end{center}

\begin{abstract}
Let $\mathbb{R}_+=[0,\infty)$ and let $A\subseteq\mathbb{R}^n_+$. We
have found the necessary and sufficient conditions under which a
function  $\Phi:A\to\mathbb{R}_+$ has an isotone subadditive
continuation on $\mathbb{R}^n_+$. It allows us to describe the
metrics, defined on the Cartesian product $X_1\times\cdots\times
X_n$ of given metric spaces $(X_1,d_{X_1}),...,(X_n,...,d_{X_n})$,
generated by the isotone metric preserving functions on
$\mathbb{R}^n_+$. It also shows that the isotone metric preserving
functions $\Phi:\mathbb{R}^n_+\to\mathbb{R}_+$ coincide with the
first moduli
 of continuity of the nonconstant bornologous functions
 $g:\mathbb{R}^n_+\to\mathbb{R}_+$. We discuss some
 algebraic properties of sets $X\subseteq \mathbb{R}$ providing the
 existence of isometric embeddings $f:B\to X$ for every three-point  $B\subseteq
 \mathbb{R}$. In particular, we prove that every finite subset of
 $\mathbb{R}$ is isometric to some subset of transcendental real
 numbers.
\end{abstract}

\bigskip
{\bf Mathematics Subject Classification (2010):} Primary 54E35, Secondary 26A15, 26B25.

\bigskip
{\bf Key words:} metric product, isotone metric preserving function
of several variables, bornologous function, isotone subadditive
function, modulus of continuity.

\section{Introduction}
\normalsize Let $n$ be a positive integer number.
\begin{definition}\label{def1.1}
Let $(X_1,d_{X_1}),...,(X_n,d_{X_n})$ be metric spaces and let
\mbox{$P=X_1\times\cdots\times X_n$} be the Cartesian product of
$X_1,...,X_n$. A metric $d$ on $P$ is a \emph{metric product} if
there exists a function $\Phi:\mathbb{R}^n_+\to\mathbb{R}_+$ such
that the equality
\begin{equation}\label{eq1.1}
d((x_1,...,x_n),(y_1,...,y_n))=\Phi(d_{X_1}(x_1,y_1),...,d_{X_n}(x_n,y_n))
\end{equation}
holds for all  $(x_1,...,x_n)$, $(y_1,...,y_n)\in P$.
\end{definition}

\begin{definition}\label{def1.2}
A function $\Phi:\mathbb{R}^n_+\to\mathbb{R}_+$ is a \emph{metric
preserving} function (of $n$ variables) if
$\Phi(d_{X_1}(\cdot,\cdot),...,d_{X_n}(\cdot,\cdot))$ is a metric on
$X_1\times \cdots\times X_n$ for every collection of metric spaces
$(X_1,d_{X_1}),...,(X_n,d_{X_n})$.
\end{definition}

There were  published in recent decades a lot of works  dealing with
metric products. The Euclidean and the Minkowski ranks for arbitrary
metric spaces and their behavior with respect to products were
studied in ~\cite{BFS} and ~\cite{FS}. The papers ~\cite{M},
~\cite{T} and ~\cite{AF} concern the long-standing problem of S.
Ulam on the uniqueness of decomposition of metric spaces into metric
products of indecomposable factors. The search of conditions for
multiplicativity of metric properties in terms of properties of the
functions $\Phi$ is another natural problem arising under studies of
metric products (see, for  example, ~\cite{Te2}, ~\cite{HM},
~\cite{OS}, and ~\cite{DM2} for some results in this direction). The
metric products of an arbitrary infinite or finite family of metric
spaces were studied in ~\cite{BD}.

In the case where the number of factors in the Cartesian product equals one we obtain
 a new metric $\Phi \circ d_1$. The metrics of such form occupy a special
position among metric products (see ~\cite{Co}, ~\cite{Do}, ~\cite{Sr} and ~\cite{Te}
for the surveys on the metric preserving functions of one variable).

The present paper deals with the isotone metric preserving functions
and isotone metric products. In particular we characterize the
metrics which are isotone, amenable or subadditive  metric products
(see section 3 for the exact definitions and formulations of
results). These characterizations are based on the construction of
continuation of isotone amenable or isotone subadditive functions
defined on subsets of $\mathbb{R}^n_+$ (see section 2 which play an
auxiliary role in the paper).

In section 4 we prove that the isotone metric preserving functions
coincide with the moduli of continuity of nonconstant bornologous
functions $f:\mathbb{R}^n_+\to\mathbb{R}_+$. It is a generalization
of similar one-dimensional result from~\cite{DM}. The fifth final
section of the paper has the origin in an interesting observation of
I. Herbut and M. Moszy\`{n}ska that to establish some condition of
multiplicativity it often suffices to examine the metric products on
$\mathbb{R}^2\times\mathbb{R}^2$ with the usual Euclidean metrics in
the factors (see\cite{HM}). The main reason of this phenomenon is
the following universal property: for every metric triangle, there
is an isometric embedding in $\mathbb{R}^2$. The metric spaces which
are favorable for isometric embeddings of the triangles situated in
$\mathbb{R}=(-\infty,\infty)$ play a similar role for metric
preserving, isotone, functions (see Theorem \ref{th5.1}). Examples
of such spaces can be found in section 5.

\section{Continuation of isotone functions}

\noindent Let $\overline{x}=(x_1,\dots,x_n)$,
$\overline{y}=(y_1,\dots,y_n)$ be two vectors from $\mathbb{R}^n_+$.
Then we write:
  \begin{equation}\label{eq2.1}
  \begin{split}
  &\mbox{(i)  } \overline{x} \leqslant \overline{y} \Leftrightarrow x_i\leqslant y_i \mbox{ for every } i\in \{1,...,n\}; \\
  &\mbox{(ii)  }   \overline{x} < \overline{y} \Leftrightarrow \overline{x}\leqslant \overline{y} \mbox{ and } \overline{x}\neq\overline{y}; \\
  &\mbox {(iii)  } |\overline{x} - \overline{y}| := (|x_1 - y_1|,\dots,|x_n - y_n|).
  \end{split}
  \end{equation}

We need the following definition.
\begin{definition}\label{def2.1}
Let $A\subseteq \mathbb{R}^n_+$. A function $f:A\to \mathbb{R}_+$ is
isotone if the implication
$$
(\bar{x}\leqslant \bar{y}) \Rightarrow (f(\overline{x})\leqslant
f(\overline{y}))
$$
holds for all $\bar{x},\bar{y}\in A$.
\end{definition}
It is easy to see that $f:\mathbb{R}^n_+\to\mathbb{R}_+$ is isotone
if and only if $f$ is increasing by coordinates, i.e., $f$ is
separately increasing in every variable.

For $\bar{a} \in \mathbb{R}^n_+$  we denote by $\bar{a}^{\nabla}$
the lower cone of $\bar{a}$ in the partially ordered set
$(\mathbb{R}^n_+,\leqslant)$, i.e.,
$$
\bar{a}^{\nabla}=\{\bar{x}\in\mathbb{R}^n_+:\bar{x}\leqslant\bar{a}\}.
$$
 The following lemma gives us the necessary and sufficient conditions
 under which a ``partially defined'' isotone function is a
 restriction of a ``completely defined'' isotone function.
 \begin{lemma}\label{lem2.2}
 Let $A$ be a subset of $\mathbb{R}^n_+$ and let
 $\Phi:A\to\mathbb{R}_+$ be an isotone function. The following
 conditions are equivalent.
\begin{itemize}
  \item [(i)] There is an isotone function
  $\Psi:\mathbb{R}^n_+\to\mathbb{R}_+$ such that
  \begin{equation}\label{eq2.2}
  \Psi|_A=\Phi
  \end{equation}
  where $\Psi|_A$ is the restriction of $\Psi$ on the  set $A$.
  \item [(ii)] The inequality
  \begin{equation}\label{eq2.3}
    \sup\{\Phi(\bar{x}):\bar{x}\in A\cap \bar{b}^{\nabla}\}<\infty
  \end{equation}
holds for every $\bar{b}\in\mathbb{R}^n_+$.
\end{itemize}
 \end{lemma}
\begin{proof}
The implication (i)$\Rightarrow$(ii) is trivial. Suppose that
condition (ii) is fulfilled. Let us define
  \begin{equation}\label{eq2.4}
\Phi^*(\bar{y}) := \sup\{\Phi(\bar{x}):\bar{x}\in A\cap
\bar{y}^{\nabla}\}, \quad \bar{y} \in\mathbb{R}^n_+.
  \end{equation}
Inequality~(\ref{eq2.3}) implies the double inequality
  \begin{equation}\label{eq2.5}
0\leqslant\Phi^*(\bar{a})<\infty
  \end{equation}
for every $\bar{a}\in\mathbb{R}^n_+$. (In particular we take
$\Phi(\bar{a})=\sup \varnothing =0$ if $\bar{a}^{\nabla}\cap
A=\varnothing$.) Hence $\Phi^*$ is a nonnegative function on
$\mathbb{R}^n_+$. Since the inclusion $\bar{a}^{\nabla}\cap
A\subseteq \bar{b}^{\nabla}\cap A$ holds for $\bar{a}\leqslant
\bar{b}$, the function $\Phi^*$ is isotone. Definition~\ref{def2.1}
and ~(\ref{eq2.4}) imply $\Phi(\bar{a})=\Phi^*(\bar{a})$ for
$\bar{a}\in A$. Consequently ~(\ref{eq2.2}) holds with
$\Psi=\Phi^*$. Condition (i) follows.
\end{proof}
\begin{remark}\label{rem2.3}
 It is easy to prove the inequality
$$
\Psi(\bar{x}) \geqslant \Phi^*(\bar{x})
$$
for every $\bar{x}\in\mathbb{R}^n_+$ and every isotone function
$\Psi:\mathbb{R}^n_+\to\mathbb{R}_+$ satisfying~(\ref{eq2.2}).
\end{remark}
\begin{definition}\label{def2.4}
Let $A\subseteq\mathbb{R}^n_+$ and let $\bar{0}=(0,...,0)\in A$. We
shall say that a function $\Phi:A\to\mathbb{R}^n_+$ is
\emph{amenable} if $\Phi(\bar{0})=0$ and $\Phi(\bar{x})>0$ for every
$\bar{x}\in A\backslash\{\bar{0}\}$.
\end{definition}
For $n=1$ and $A=\mathbb{R}_+$, Definition~\ref{def2.4} turns into
the definition of amenable functions from~\cite{Dob}.

 Let $t \in (0,\infty)$ and $j_0\in\{1,...,n\}$. Define the
vectors $\bar{t}^{j_0}=(t_1^{j_0},...,t_n^{j_0})\in \mathbb{R}^n_+$
by the rule
\begin{equation}\label{eq2.6}
t_i^{j_0}:=
\begin{cases}
t \quad\text{if } i=j_0, \\
0 \quad\text{if } i\neq j_0.
\end{cases}
\end{equation}
\begin{lemma}\label{lem2.5}
Let $A$ be a subset of $\mathbb{R}^n_+$, $\bar{0}\in A$ and let
$\Phi:A\to\mathbb{R}_+$ be an isotone amenable function. The
function $\Phi^*:\mathbb{R}^n_+\to\mathbb{R}_+$ defined
by~(\ref{eq2.4}) is amenable if and only if for every
$j_{0}\in\{1,...,n\}$ and $\varepsilon \in (0,\infty)$ there is
$t=t(\varepsilon,j_0)\in(0,\infty)$ such that
\begin{equation}\label{eq2.7}
t\leqslant\varepsilon \mbox{ and }\bar{t}^{j_0}\in A.
\end{equation}
\end{lemma}
\begin{proof}
Suppose that for every $j_0\in\{1,...,n\}$ and $\varepsilon\in
(0,\infty)$ there is $t=t(\varepsilon,j_0)\in (0,\infty)$ such
that~(\ref{eq2.7}) holds. It was shown in the proof of
Lemma~\ref{lem2.2} that $\Phi^*(\bar{0})=\Phi(\bar{0})=0$ and
$0\leqslant \Phi^*(\bar{x})<\infty$ for $\bar{x}\in\mathbb{R}^n_+$.
Consequently $\Phi^*$ is amenable if and only if
\begin{equation}\label{eq2.8}
\Phi(\bar{y})>0
\end{equation}
for $\bar{y}>\bar{0}$. Now we prove inequality~(\ref{eq2.8}).
By~(\ref{eq2.1}) the inequality $\bar{y}>\bar{0}$ holds if and only
if there are $\varepsilon>0$ and $j_0\in\{1,...,n\}$ such that
$y_{j_0}\geqslant \varepsilon$. By the supposition there is
$\bar{t}^{j_0}\in A$ for which
$$
\bar{0}<\bar{t}^{j_0}\leqslant \bar{y}.
$$
Consequently $\bar{t}^{j_0}\in A\cap \bar{y}^{\nabla}$. Hence
$$
\Phi^*(\bar{y})=\sup\{\Phi(\bar{x}):\bar{x}\in
A\cap\bar{y}^{\nabla}\}\geqslant \Phi(\bar{t}^{j_0})
$$
Since $\Phi$ is amenable and $\bar{t}^{j_0}>\bar{0}$ we have
$\Phi(\bar{t}^{j_0})>0$. Inequality~(\ref{eq2.8}) follows.

Conversely, suppose  that $\Phi^*(y)$ is amenable but there is
$j_0\in\{1,...,n\}$ and $\varepsilon>0$ such that
$$
\bar{t}^{j_0}\notin A
$$
for every $t\in (0,\varepsilon]$. Then
$A\cap\bar{\varepsilon}^{j_0\nabla}=\{\bar{0}\}$, so that
$$
\Phi^*(\bar{\varepsilon}^j)=\sup\{\Phi(\bar{x}):\bar{x}\in\{\bar{0}\}\}=\Phi(\bar{0})=0.
$$
Hence $\Phi^*$ is not amenable, contrary to the supposition.
\end{proof}

\begin{remark}\label{rem2.6}
The conclusion of Lemma~\ref{lem2.5} can be reformulated by the
following way.

\begin{itemize}
  \item The function $\Phi^*:A\to\mathbb{R}_{+}$ is amenable if and
  only if $A-\{\bar{0}\}$ is a coinitial subset of
  $\mathbb{R}^n_+-\{\bar{0}\}$.
\end{itemize}
\end{remark}

For every $\bar{a}=(a_1,...,a_n)\in\mathbb{R}^n_+$ and
$j\in\{1,...,n\}$ define $\pr_j(\bar{a})=a_j$ and denote by
$\bar{a}^{\Delta}$ the upper cone of $\bar{a}$, i.e.,
$$
\bar{a}^{\Delta} =
\{\bar{x}\in\mathbb{R}^n_+:\bar{x}\geqslant\bar{a}\}.
$$
\begin{lemma}\label{lem2.6}
Let $A$ be a subset of $\mathbb{R}^n_+$, $\bar{0}\in A$ and let
$\Phi:A\to\mathbb{R}_+$ be an isotone and amenable. The following
conditions are equivalent
\begin{itemize}
  \item [(i)] There is an isotone amenable function
  $\Psi:\mathbb{R}^n_+\to\mathbb{R}_+$ such that
  $$
\Psi|_{A}=\Phi.
  $$
  \item [(ii)] We have inequality~(\ref{eq2.3}) for every
  $\bar{b}\in\mathbb{R}^n_+$ and the equalities
  $$
\inf (\pr_{j}(B))=0, \quad j=1,...,n,
  $$
  hold for every $B\subseteq A$ with $\inf (\Phi(B))=0$.
\end{itemize}
\end{lemma}
\begin{proof}
\textbf{(i)}$\Rightarrow$\textbf{(ii)} Suppose (i) holds. By
Lemma~\ref{lem2.2} inequality~(\ref{eq2.3}) holds for every
$\bar{b}\in\mathbb{R}^n_+$. Assume that there are $B\subseteq A$ and
$j_0\in\{1,...,n\}$ with
\begin{equation}\label{eq2.9}
\inf (\Phi(B))=0
\end{equation}
and
\begin{equation}\label{eq2.10}
t=\inf (\pr_j(B))>0
\end{equation}
From~(\ref{eq2.9}) and~(\ref{eq2.10}) the inequality $t>\infty$
follows. Indeed, if $t=+\infty$, then $B=\varnothing$, so that $\inf
(\Phi(B))=+\infty$, contrary to~(\ref{eq2.9}). Let $\bar{t}^{j_0}$
be the vector defined by~(\ref{eq2.6}). Then the inequality
$\bar{t}^{j_0}\leqslant \bar{b}$ holds for every $\bar{b}\in B$.
Since $\Psi$ is amenable and isotone and $\bar{t}^{j_0}>\bar{0}$, we
obtain
$$
0<\Psi(\bar{t}^{j_0})\leqslant \inf (\Psi(B))=\inf (\Phi(B)),
$$
contrary to~(\ref{eq2.9}). This contradiction shows that the
implication (i)$\Rightarrow$(ii) is true.

\textbf{(ii)}$\Rightarrow$\textbf{(i)} Suppose that condition (ii)
is fulfilled. Consider first the case  when for every
$j_0\in\{1,...,n\}$ there is $\bar{x}\in A$ such that
\begin{equation}\label{eq2.11}
\pr_{j_0}(\bar{x})>0.
\end{equation}
Write
$$
\tau_j:=\sup(\pr_j A), \quad j\in \{1,...,n\}.
$$
Inequality~(\ref{eq2.11}) implies that $\tau_j>0$ for every $j$. Let
us continue $\Phi$ from $A$ to the set
\begin{equation}\label{eq2.12}
\Gamma =A \cup (\bigcup\limits_{j=1}^n\{\bar{t}^j:t\in(0,\tau_j)\})
\end{equation}
by the rule
\begin{equation}\label{eq2.13}
\Phi^*(\bar{a})=
\begin{cases}
\Phi(\bar{a}) \quad\text{if } a\in A, \\
\inf\{\Phi(\bar{x}):\bar{x}\in A\cap\bar{a}^{\Delta}\} \quad\text{if
} \bar{a} \in \bigcup\limits_{j=1}^n\{\bar{t}^j:t\in(0,\tau_j)\}.
\end{cases}
\end{equation}
If $\bar{a}\in
A\cap(\bigcup\limits_{j=1}^n\{\bar{t}^j:t\in(0,\tau_j)\})$, then
$\bar{a}\in A\cap\bar{a}^{\Delta}$ and, since $\Phi$ is isotone,
$\Phi(\bar{x})\geqslant \Phi(\bar{a})$ for every $\bar{x}\in
A\cap\bar{a}^{\Delta}$. Thus we obtain
$$
\Phi^*(\bar{a})=\inf \{\Phi(\bar{x}):\bar{x}\in A\cap
\bar{a}^{\Delta}\}=\Phi(a)
$$
for every $\bar{a}\in A\cap(\bigcup\limits_{j=1}^n\{\bar{t}^j:t\in
(0,\tau_j)\})$, i.e., $\Phi^*$ is correctly defined. Using the
isotonicity of $\Phi$ and the definition of $\Phi^*$ we can easily
show that $\Phi^*(\bar{x})$ is finite for every $\bar{x}\in \Gamma$
and $\Phi^*:\Gamma\to\mathbb{R}_+$ is isotone. To prove that
$\Phi^*$ is amenable suppose that there is $\bar{\gamma}\in \Gamma$
such that $\Phi^*(\bar{\gamma})=0$. Since $\Phi$ is amenable, the
last equality shows
\begin{equation}\label{eq2.14}
\bar{\gamma}\in \bigcup\limits_{j=1}^n\{\bar{t}^j:t\in(0,\tau_j)\}.
\end{equation}
Consequently
\begin{equation}\label{eq2.15}
0=\Phi^*(\bar{\gamma})=\inf \{\Phi(\bar{x}):\bar{x}\in A\cap
\bar{\gamma}^{\Delta}\} = \inf (\Phi(A\cap\bar{\gamma}^{\Delta})).
\end{equation}
Membership relation~(\ref{eq2.14}) implies that there are
$j_0\in\{1,...,n\}$ and $t\in (0,\tau_{j_0})$ such that
$\bar{\gamma} = \bar{t}^{j_0}$. Condition (ii) and~(\ref{eq2.15})
give us the equality
$$
\inf (\pr_{j_0}(A\cap\bar{t}^{j_0\Delta}))=0.
$$
In particular there is $\bar{x}\in (A\cap\bar{t}^{j_0\Delta})$ such
that
\begin{equation}\label{eq2.16}
\pr_{j_0}(\bar{x})<t.
\end{equation}
By definition of the upper cone  we have $\pr_j(\bar{x})\geqslant
\pr_j(\bar{t}^{j_0})$ for every $\bar{x}\in \bar{t}^{j_0\Delta}$ and
$j\in \{1,...,n\}$. Consequently $\pr_{j_0}(\bar{x})\geqslant
\pr_{j_0}(\bar{t}^{j_0})=t$, contrary to~(\ref{eq2.16}). Thus
$\Phi^*$ is amenable. Consider now the case where there is
$j\in\{1,...,n\}$ such that $\pr_j(A)=\{0\}$. We may assume, after a
suitable permutation, that there is $j_0\in\{1,...,n\}$ such that
$$
\tau_{j} = \sup (\pr_j(A)) = 0
$$
if and only if $1\leqslant j \leqslant j_0$. Let us continue $\Phi$
from $A$ to the set
$$
A_0 = A\cup(\bigcup\limits_{j=1}^{j_0}\{\bar{t}^j:t\in(0,\infty)\})
$$
as
\begin{equation*}
\Phi_0(\bar{a})=
\begin{cases}
\Phi(\bar{a}) \quad\text{if } a\in A, \\
t \quad\text{if } \bar{a} =\bar{t}^j \text{ with } j=\{1,...,j_0\}.
\end{cases}
\end{equation*}
It is easy to prove that  $\Phi_0:A_0\to\mathbb{R}_+$ is an isotone
and amenable continuation of $\Phi:A\to\mathbb{R}_+$. Now for every
$j\in \{1,...,n\}$ there is $\bar{x}\in A_0$ such that
inequality~(\ref{eq2.11}) holds. Consequently there is a
continuation $\Phi^*$ of $\Phi^0$ on the set $\Gamma$
(see~(\ref{eq2.12})) defined as in~(\ref{eq2.13}). To find a
continuation of $\Phi^*:\Gamma\to\mathbb{R}_+$  to an isotone
amenable function $\Psi:\mathbb{R}^n_+\to\mathbb{R}_+$ it suffices
to note that for every $j\in\{1,...,n\}$ and every $\varepsilon \in
(0,\infty)$ there is $t=t(\varepsilon,j)\in(0,\infty)$ such that
$t\leqslant \varepsilon$ and $\bar{t}^j\in\Gamma$. Hence a desirable
continuation can be obtained as in Lemma~\ref{lem2.5}.
\end{proof}
Recall that a function $\Phi:\mathbb{R}^n_+\to\mathbb{R}_+$ is
subadditive if the inequality
\begin{equation}\label{eq2.17}
\Phi(\bar{x}+\bar{y})\leqslant\Phi(\bar{x})+\Phi(\bar{y})
\end{equation}
holds for all $\bar{x}, \bar{y}\in \mathbb{R}^n_+$.

Now we shall give an extension of this property to the case of
isotone functions defined on arbitrary subsets of $\mathbb{R}^n_+$.
\begin{definition}\label{def2.7}
Let $A$ be a subset of $\mathbb{R}^n_+$. An isotone function
$\Phi:A\to\mathbb{R}_+$ is subadditive if the implication
\begin{equation}\label{eq2.18}
(\bar{x}\leqslant\sum\limits_{i=1}^{m}\bar{x}^{i})\Rightarrow
(\Phi(\bar{x})\leqslant\sum\limits_{i=1}^{m}\Phi(\bar{x}^{i}))
\end{equation}
holds for all $\bar{x}$, $\bar{x}^{1},...,\bar{x}^{m}\in A$ and
every positive integer number $m\geqslant 2$.
\end{definition}
\begin{remark}\label{rem2.7*}
It is easy to prove that a function $\Phi:A\to\mathbb{R}_+$ is
isotone and subadditive if and only if the
implication~(\ref{eq2.18}) holds for every $\bar{x}\in A$ and all
$\bar{x}^{1},...,\bar{x}^{m} \in A$ and every integer number
$\mathbf{m\geqslant 1}$.
\end{remark}

\begin{remark}\label{rem2.8}
If $\Phi:\mathbb{R}^n_+\to\mathbb{R}_+$ is isotone,
then~(\ref{eq2.17}) holds for all $\bar{x},\bar{y}\in\mathbb{R}^n_+$
if and only if ~(\ref{eq2.18}) is true for all $\bar{x}$,
$\bar{x}^{1},...,\bar{x}^{m}\in \mathbb{R}^n_+$ with $m\geqslant 2$.
Thus Definition~\ref{def2.7} is an equivalent to the usual
definition of subadditivity if $A=\mathbb{R}^n_+$ and $\Phi$ is
isotone.
\end{remark}
\begin{lemma}\label{lem2.9}
Let $A$ be a nonempty subset of  $\mathbb{R}^n_+$. The following
conditions  are equivalent for every function
$\Phi:A\to\mathbb{R}_+$.
\begin{itemize}
  \item [(i)] The function $\Phi$ is isotone and subadditive.
  \item [(ii)] There is an isotone and subadditive function $\Psi:\mathbb{R}^n_+\to\mathbb{R}_+$
such that $\Psi|_{A}=\Phi$.
\end{itemize}
\end{lemma}
\begin{proof}
The implication (ii)$\Rightarrow$(i) follows directly from
Remark~\ref{rem2.8}. Suppose now that $\Phi$ is isotone and
subadditive. We shall construct an isotone subadditive function
$\Psi:\mathbb{R}^n_+\to\mathbb{R}_+$ such that $\Psi|_{A}=\Phi$. For
every $\bar{x}\in\mathbb{R}^n_+$ define the subset
$\textbf{S}(\bar{x})=\textbf{S}(\bar{x},A)$ of the set
$\bigcup\limits_{k=1}^{\infty}A^k$, where $A^1=A$, $A^2=A\times A$,
$A^3=A\times A\times A$ and so on, by the rule
\begin{itemize}
  \item an element $(\bar{x}^1,...,\bar{x}^k)$ of the set $\bigcup\limits_{k=1}^{\infty}A^k$
  belongs to $\textbf{S}(\bar{x})$ if and only if
\begin{equation}\label{eq2.19}
\bar{x}\leqslant\sum\limits_{i=1}^k\bar{x}^i.
\end{equation}
\end{itemize}
First consider the case when $\textbf{S}(\bar{x})\neq \varnothing$
for every $\bar{x}\in\mathbb{R}^n_+$. Define the function
$\Psi:\mathbb{R}^n_+\to\mathbb{R}_+$ as
\begin{equation}\label{eq2.20}
\Psi(\bar{x}):=\inf\{\sum\limits_{i=1}^m\Phi(\bar{x}^i):(\bar{x}^1,...,\bar{x}^m)\in
\textbf{S}(\bar{x})\}
\end{equation}
(We have $\Psi(\bar{x})\in \mathbb{R}_+$ for every
$\bar{x}\in\mathbb{R}^n_+$ because the relation
$\textbf{S}(\bar{x})\neq \varnothing $ implies the inequality
$\Psi(\bar{x})<+\infty$.) Let us show that $\Psi$ gives us the
wanted continuation  of $\Phi$. Let
$\bar{x},\bar{y}\in\mathbb{R}^n_+$ and $\bar{y}\leqslant\bar{x}$.
The inequality $\bar{y}\leqslant\bar{x}$ and~(\ref{eq2.19}) imply
$\bar{y}\leqslant\sum\limits_{i=1}^m\bar{x}^i$. Hence we obtain the
inclusion $\textbf{S}(\bar{y})\supseteq \textbf{S}(\bar{x})$ for
$\bar{y}\leqslant \bar{x}$. Using this inclusion and~(\ref{eq2.20})
we see that the implication
$$
(\bar{y}\leqslant \bar{x})\Rightarrow (\Psi(\bar{y})\leqslant
\Psi(\bar{x}))
$$
holds. Thus $\Psi$ is isotone.

To prove the subadditivity of $\Psi$ consider arbitrary $\bar{x},
\bar{y}\in\mathbb{R}^n_+$. It follows from~(\ref{eq2.20}) that for
every $\varepsilon>0$ there are
$$
(\bar{x}^1,...,\bar{x}^m)\in \textbf{S}(\bar{x}) \mbox{ and }
(\bar{y}^1,...,\bar{y}^l)\in \textbf{S}(\bar{y})
$$
such that
\begin{equation}\label{eq2.22*}
\Psi (\bar{x})+\varepsilon\geqslant
\sum\limits_{i=1}^m\Phi(\bar{x}^i)\mbox{ and } \Psi
(\bar{y})+\varepsilon\geqslant \sum\limits_{i=1}^l\Phi(\bar{y}^i).
\end{equation}
 Let $\bar{z}=\bar{x}+\bar{y}$. Then we
have
\begin{equation}\label{eq2.21}
\bar{z}\leqslant \sum\limits_{i=1}^m \bar{x}^i+\sum\limits_{i=1}^l
\bar{y}^i
\end{equation}
Define $\bar{z}^i$, $i=1,...,m+l$ by the rule:
\begin{equation}\label{eq2.22}
\bar{z}^i=\left\{
            \begin{array}{ll}
              \bar{x}^i & \hbox{if  } 1\leqslant i\leqslant m; \\
              \bar{y}^{i-m} & \hbox{if  } m+1 \leqslant i\leqslant m+l.
            \end{array}
          \right.
\end{equation}
Inequality~(\ref{eq2.21}) shows that
$(\bar{z}^1,...,\bar{z}^m,\bar{z}^{m+1},...,\bar{z}^{m+l})\in
\textbf{S}(\bar{z})$. Inequalities ~(\ref{eq2.22*}),~(\ref{eq2.22})
and ~(\ref{eq2.20}) imply
$$
\Psi (\bar{z})\leqslant \sum\limits_{i=1}^{l+m}\Phi(\bar{z}^i)=
\sum\limits_{i=1}^{m}\Phi(\bar{x}^i)+
\sum\limits_{i=1}^{l}\Phi(\bar{y}^i)\leqslant \Psi
(\bar{x})+\Psi(\bar{y})+2\varepsilon.
$$
Letting $\varepsilon \to 0$, we have
$$
\Psi(\bar{z})\leqslant \Psi(\bar{x})+\Psi(\bar{y}).
$$
Thus $\Psi$ is subadditive.

It remains to verify that $\Psi|_{A}=\Phi$. Let $\bar{x}\in A$ and
$(\bar{x}^1,...,\bar{x}^k) \in \textbf{S}(\bar{x})$. Then,
by~(\ref{eq2.19}), we have  $\bar{x}\leqslant \sum\limits_{i=1}^k
\bar{x}^i$. If $k=1$, then $\bar{x}\leqslant\bar{x}^1$, so that
$\Phi(\bar{x})\leqslant \Phi(\bar{x}^1)$ because $\Phi$ is isotone.
If $k\geqslant 2$, then using the subadditivity of $\Phi$ we also
have the inequality $\Phi(\bar{x})\leqslant
\sum\limits_{i=1}^k\Phi(\bar{x}^i)$. Consequently the last
inequality holds for every $(\bar{x}^1,...,\bar{x}^k)\in
\textbf{S}(\bar{x})$. The inequality
\begin{equation}\label{eq2.23}
\Phi(\bar{x})\leqslant\Psi(\bar{x})
\end{equation}
follows. To prove the converse inequality
\begin{equation}\label{eq2.24}
\Psi(\bar{x})\leqslant\Phi(\bar{x})
\end{equation}
note that the point $\bar{x}^1=\bar{x}$, belongs to
$\textbf{S}(\bar{x})$ for every $\bar{x}\in A$.
Consequently~(\ref{eq2.24}) follows from~(\ref{eq2.20}).
Inequalities~(\ref{eq2.23}) and ~(\ref{eq2.24}) imply that
$\Phi(\bar{x})=\Psi(\bar{x})$ for every $\bar{x}\in A$, i.e.,
$\Psi|_{A}=\Phi$.

Consider now the case when there is
$\bar{a}=(a_1,...,a_n)\in\mathbb{R}^n_+$ such that
$\textbf{S}(\bar{a})=\varnothing$. This equality and the definition
of the set $\textbf{S}(\bar{x})$ imply that there exists
$j_0\in\{1,...,n\}$ such that
\begin{equation}\label{eq2.25}
\forall \bar{y} \in A \quad \pr_{j_0}(\bar{y})=0 \mbox{ but }
a_{j_0}=\pr_{j_0}(\bar{a})>0.
\end{equation}
Let us denote by $J_0$ the set of all elements $j_0$ of
$\{1,...,n\}$ satisfying condition~(\ref{eq2.25}). For every $t\in
(0,\infty)$ and $j_0\in J_0$ write $\bar{t}^{j_0}$ for the vector
defined by~(\ref{eq2.6}). Let us continue $\Phi$ from the set $A$ to
the set
\begin{equation}\label{eq2.25*}
\Gamma := A \cup (\bigcup\limits_{j\in
J_0}\{\bar{t}^j:t\in(0,\infty)\})
\end{equation}
by the rule
\begin{equation}\label{eq2.26}
\Phi^{\circ}(\bar{a})=\left\{
            \begin{array}{ll}
              \Phi(\bar{a}) & \hbox{if  } \bar{a}\in A; \\
              c & \hbox{if  } \bar{a} \in \Gamma\backslash A.
            \end{array}
          \right.
\end{equation}
where $c$ is an arbitrary point from $\mathbb{R}_+$. It is easy to
see that
$$
\Phi^{\circ}|_{A}=\Phi \mbox{ and } \textbf{S}(\bar{x},\Gamma)\neq
\varnothing
$$
for every $\bar{x}\in \mathbb{R}^n_+$. Hence if $\Phi^{\circ}$ is
isotone and subadditive, then replacing in~(\ref{eq2.20}) $\Phi$ by
$\Phi^{\circ}$ we can finish the proof.

In accordance with Remark~\ref{rem2.7*}, $\Phi^{\circ}$ is isotone
and subadditive if implication~(\ref{eq2.18}) holds for all
$\bar{x}, \bar{x}^1,...,\bar{x}^m\in \Gamma$ and every $m \geqslant
1$. It is clear that~(\ref{eq2.18}) holds  whenever
$\bar{x}=\bar{0}$.  To clarify the proof of~(\ref{eq2.18}) for
$\bar{x}>\bar{0}$ first consider the case $m=1$. Note that every
nonzero $\bar{x}\in A$ and nonzero $\bar{y}\in \Gamma\backslash A$
are incomparable, i.e., we have neither $\bar{x}\leqslant \bar{y}$
nor $ \bar{y} \leqslant \bar{x}$. Consequently if
$\bar{0}<\bar{x}\leqslant \bar{x}^1$, then either $\bar{x},
\bar{x}^1 \in A$ or $\bar{x},\bar{x}^1\in \Gamma \backslash A$. The
function $\Phi^{\circ}|_{A}=\Phi$ is isotone and subadditive by the
supposition. The function $\Phi^{\circ}|_{\Gamma \backslash A}$ is
also isotone and subadditive as a constant function.
Consequently~(\ref{eq2.18}) holds when $m=1$, i.e., $\Phi^{\circ}$
is isotone. Let us consider an arbitrary $m\geqslant 1$. Suppose
that $\bar{0}<\bar{x}\leqslant \sum\limits_{i=1}^m \bar{x}^i$,
$\bar{x}\in A$ and $\bar{x}^1,...,\bar{x}^m \in \Gamma$. Define a
subset $I$ of the set $\{1,...,m\}$ by the rule
$$
(i\in I)\Leftrightarrow (i\in \{1,...,m\} \mbox{ and } \bar{x}^i\in
A).
$$
Then, using the definition of the set $\Gamma$, we obtain the
inequality $\bar{x}\leqslant \sum\limits_{i\in I}\bar{x}^i$. (Note
that the implication
$$
(\bar{x} \in \Gamma\backslash A \mbox{ and }\bar{0}<\bar{x}\leqslant
\sum\limits_{i=1}^m \bar{x}^i)\Rightarrow (\bar{x}\leqslant
\sum\limits_{i\in I}\bar{x}^i)
$$
can be considered as a generalization of incomparability of nonzero
$\bar{x}\in A$ and $\bar{y}\in \Gamma \backslash A$.) Since
$\bar{x}^i\in A$ for $i\in I$, the inequality $\bar{x}\leqslant
\sum\limits_{i\in I}\bar{x}^i$ implies
$$
\Phi^{\circ}(x)=\Phi(\bar{x})\leqslant \sum\limits_{i\in
I}\Phi(\bar{x}^i)= \sum\limits_{i\in
I}\Phi^{\circ}(\bar{x}^i)\leqslant
\sum\limits_{i=1}^{m}\Phi^{\circ}(\bar{x}^i).
$$
Similarly we can show the inequality
$$
\Phi^{\circ}(\bar{x}^i)\leqslant \sum\limits_{i=1}^m
\Phi^{\circ}(\bar{x}^i)
$$
 for $\bar{x}\in \Gamma\backslash A$ and
$\bar{x}^1,...,\bar{x}^m\in \Gamma$. Thus~(\ref{eq2.18}) holds for
all $\bar{x}, \bar{x}^1,...,\bar{x}^m \in \Gamma$.
\end{proof}
We finish this section by the following proposition, which is
interesting in its own right.
\begin{theorem}\label{th2.10}
Let $A$ be a subset of $\mathbb{R}^n_+$ such that for every  $j\in
\{1,...,n\}$ there is $\bar{a}\in A$ satisfying the inequality
$\pr_{j}(\bar{a})>0$. Then, for every $\Phi:A\to\mathbb{R}_+$, the
function $\Psi:\mathbb{R}^n_+\to\mathbb{R}_+$ defined
by~(\ref{eq2.20}) has the following properties.
\begin{itemize}
  \item [(i)] $\Psi$ is isotone and subadditive.
  \item [(ii)] The inequality $\Phi(\bar{x})\geqslant\Psi(\bar{x})$
  holds for every $\bar{x}\in A$.
  \item [(iii)] The equality $\Phi=\Psi|_{A}$ holds if and only if
  $\Phi$ is isotone and subadditive.
  \item [(iv)] If $F:\mathbb{R}^n_+\to\mathbb{R}_+$ is isotone
  subadditive function such that the inequality
  $\Phi(\bar{x})\geqslant F(\bar{x})$ holds for every $\bar{x}\in A$,
  then the inequality $\Psi(\bar{x})\geqslant F(\bar{x})$ also holds
  for every $\bar{x}\in \mathbb{R}^n_+$.
\end{itemize}
\end{theorem}
This theorem can be proved by modification of the proof of
Lemma~\ref{lem2.9} so it can be omitted here.

\begin{remark}\label{rem2.11}
If for given  $\Phi:A\to\mathbb{R}_+$ there is a function
$\Psi:\mathbb{R}^n_+\to\mathbb{R}_+$ meeting conditions (i)-(iv),
then for every $j\in \{1,...,n\}$ there is $\bar{a}\in A$ such that
$\pr_j(\bar{a})>0$. It can be obtained from~(\ref{eq2.26}) with
$\Phi=\Psi|_{A}$.
\end{remark}
\section{From metric products to metric preserving functions}
Let us denote by $\mathfrak{F}^n_i$ the set of isotone metric
preserving functions.
\begin{theorem}\label{th3.1}
Let $f:\mathbb{R}^n_+\to\mathbb{R}_+$ be isotone. Then $f$ belongs
to $\mathfrak{F}^n_i$ if and only if $f$ is subadditive and
amenable.
\end{theorem}
This theorem is a direct consequence
 of Theorem 2.6 from~\cite{BD} and Theorem 1 from Chapter 9 of ~\cite{Dob}.
\begin{definition}[\cite{DM2}]\label{def3.2}
 Let $(X_1,d_{X_1}),...,(X_n,d_{X_n})$ be metric spaces.
 A metric  $d$ defined on the product $X_1\times \cdots\times X_n$ is \emph{distance-increasing} if

\begin{equation}\label{eq3.1}
d((x_1,...,x_n),(y_1,...,y_n))\leqslant
d((\tilde{x}_1,...,\tilde{x}_n),(\tilde{y}_1,...,\tilde{y}_n))
\end{equation}
whenever
$$
(d_{X_1}(x_1,y_1),...,d_{X_n}(x_n,y_n))\leqslant(d_{X_1}(\tilde{x}_1,\tilde{y}_1),...,d_{X_n}(\tilde{x}_n,\tilde{y}_n)).
$$
\end{definition}

Recall that the distance set of a metric space $(X,d)$ is  the set
$$
D_X=\{d(x,y):x,y\in X\}.
$$
Let $(X_1,d_{X_1}),...,(X_n,d_{X_n})$ be metric spaces. We shall say
that a metric product $d$ is isotone (subadditive) if there is an
isotone (subadditive) function $\Phi:D_{X_1}\times\cdots\times
D_{X_n}\to\mathbb{R}_+$ such that~(\ref{eq1.1}) holds for all
$(x_1,...,x_n),(y_1,...,y_n)\in X_1\times\cdots\times X_n$.

\begin{lemma}\label{lem3.3}
Let $(X_1,d_{X_1}),...,(X_n,d_{X_n})$ be nonempty metric spaces and
let $d$ be a metric defined on $P=X_1\times\cdots\times X_n$. The
following conditions are equivalent.
\begin{itemize}
  \item [(i)] $d$ is an isotone metric product.
  \item [(ii)] $d$ is distance-increasing.
\end{itemize}
\end{lemma}
\begin{proof}
The implication (i)$\Rightarrow$(ii) is clear. Suppose that $d$ is
distance-increasing. Let us prove that there exists an isotone
$$
\Phi:D_{X_1}\times\cdots\times D_{X_n}\to \mathbb{R}_+
$$
such that~(\ref{eq1.1}) holds for all $(x_1,...,x_n),
(y_1,...,y_n)\in P$.

Let us set  $d_{X_i}(x_i,y_i)=d_{X_i}(\tilde{x}_i,\tilde{y}_i)$,
$i=1,\dots,n$, in Definition~\ref{def3.2}. Then
inequality~(\ref{eq3.1}) implies
\begin{equation*}
d((x_1,...,x_n),(y_1,...,y_n))=d((\tilde{x}_1,...,\tilde{x}_n),(\tilde{y}_1,...,\tilde{y}_n)).
\end{equation*}
Thus, the function $d:P\times P\to\mathbb{R}_+$ depends only on the
distances $d_{X_1}(x_1,y_1),...,d_{X_n}(x_n,y_n)$. Consequently,
there exists a function
$$
    \Phi:D_{X_1}\times\cdots\times D_{X_n}\to\mathbb{R}_+^n
$$
such that the following diagram
\begin{equation}\label{eq.diag}
\begin{diagram}
\node{P\times P}
    \arrow[4]{e,t}{d}
    \arrow{s,l}{I}
\node[4]{\mathbb R^+}\\
\node{(X_1\times X_1)\times\cdots\times(X_n\times X_n)}
    \arrow[4]{e,t}{d_{X_1}\otimes\cdots\otimes d_{X_n}}
\node[4]{D_{X_1}\times\cdots\times D_{X_n}}
    \arrow{n,r}{\Phi}
\end{diagram}
\end{equation}
 is commutative, where $I$ is the identification mapping,
$$
I((x_1,...,x_n),(\tilde{x}_1,...,\tilde{x}_n))=((x_1,\tilde{x}_1),...,(x_n,\tilde{x}_n)),
$$
and $ d_{X_1}\otimes\cdots\otimes d_{X_n}$ is the direct product,
$$
 d_{X_1}\otimes\cdots\otimes d_{X_n}((x_1,\tilde{x}_1),...,(x_n,\tilde{x}_n))=(d_{X_1}(x_1,\tilde{x}_1),...,d_{X_n}(x_n,\tilde{x}_n)).
$$
Using the fact that $d$ is distance-increasing it is easy to show
that $\Phi$ is isotone.
\end{proof}

Let us define the metric $\rho_{\infty}$ on $P=X_1\times\cdots\times
X_n$ as
$$
\rho_{\infty}((x_1,...,x_n),(y_1,...,y_n))=\max\limits_{1\leqslant
i\leqslant n} d_{X_i}(x_i,y_i).
$$

We shall say that an isotone metric product $d$ has an isotone
(amenable, subadditive) continuation if there is an isotone
(amenable, subadditive) function
$\Phi:\mathbb{R}^n_+\to\mathbb{R}_+$ such that
$$
d((x_1,...,x_n),(y_1,...,y_n)) =
\Phi(d_{X_1}(x_1,y_1),...,d_{X_n}(x_n,y_n))
$$
for all $(x_1,...,x_n), (y_1,...,y_n) \in P$.

 Let $(X,d)$ and $(Y,\rho)$ be metric spaces. Recall that a
mapping $f:X\to Y$ is \emph{bornologous} if for every $\varepsilon
\in \mathbb{R}_+$ there is
$\delta=\delta(\varepsilon)\in\mathbb{R}_+$ such that the
implication
$$
(d(x,y)\leqslant \varepsilon)\Rightarrow (\rho(f(x),f(y))\leqslant
\delta)
$$
holds for all $x,y \in X$ (see \cite[p.6]{Ro}).

\begin{remark}\label{rem3.4}
It is easy to prove that $f:X\to Y$ is bornologous if and only if
there is an increasing $g:\mathbb{R}_+\to\mathbb{R}_+$ such that
\begin{equation}\label{eq3.3}
\rho(f(x),f(y))\leqslant g(d(x,y))
\end{equation}
for all $x,y\in X$.
\end{remark}
\begin{theorem}\label{th3.5}
Let $(X_1,d_{X_1}),...,(X_n,d_{X_n})$ be metric spaces. The
following conditions are equivalent for every metric $d$ on
$P=X_1\times\cdots\times X_n$.
\begin{itemize}
  \item [(i)] $d$ is an isotone metric product and has an isotone continuation.
  \item [(ii)] $d$ is distance-increasing and the identical mapping
  $$
    P\ni (x_1,...,x_n)\mapsto (x_1,...,x_n)\in P
  $$
  is bornologous as a mapping from $(P,\rho_{\infty})$ to $(P,d)$.
\end{itemize}
\end{theorem}
\begin{proof}
\textbf{(i)}$\Rightarrow$\textbf{(ii)}. Let $d$ be an isotone metric
product having an isotone continuation. By Lemma~\ref{lem3.3}, the
metric $d$ is distance-increasing. Let
$\Psi:\mathbb{R}^n_+\to\mathbb{R}_+$ be an isotone function such
that
\begin{equation}\label{eq3.4}
d((x_1,...,x_n),(y_1,...,y_n))=\Psi(d_{X_1}(x_1,y_1),...,d_{X_n}(x_n,y_n))
\end{equation}
for $(x_1,...,x_n)$, $(y_1,...,y_n)\in P$. Write
\begin{equation}\label{eq3.5}
A:=D_{X_1}\times\cdots\times D_{X_n}.
\end{equation}
The restriction $\Psi|_{A}$ is an isotone function on $A\subseteq
\mathbb{R}^n_+$ and the function $\Psi$ is an isotone continuation
of $\Psi|_A$. Consequently, by Lemma~\ref{lem2.2},  the inequality
$$
\sup\{\Psi|_{A}(\bar{x}):\bar{x} \in A\cap \bar{b}^{\nabla}\}<\infty
$$
holds for every $\bar{b}\in \mathbb{R}^n_+$. Let us take here
$\bar{b}=\bar{b}_{\varepsilon}=(\varepsilon,...,\varepsilon)$ with
 $\varepsilon \in\mathbb{R}_+$. We claim that the
inequality
\begin{equation}\label{eq3.6}
d((x_1,...,x_n),(y_1,...,y_n))\leqslant
g(\rho_{\infty}((x_1,...,x_n),(y_1,..,y_n)))
\end{equation}
holds for all $(x_1,..,x_n)$, $(y_1,..,y_n)\in P$ with
\begin{equation}\label{eq3.7}
g(\varepsilon):=\sup\{\Psi|_{A}(\bar{z}):\bar{z}\in
A\cap\bar{b}^{\nabla}_{\varepsilon}\},\quad \varepsilon \in
\mathbb{R}_+.
\end{equation}
Indeed, the relation  $\bar{z}\in
A\cap\bar{b}_{\varepsilon}^{\nabla}$ holds if and only if there are
$(x_1,...,x_n)$, $(y_1,...,y_n)\in P$ such that
$d_{X_i}(x_i,y_i)=z_i$, $i=1,...,n$, and
$$
\max\limits_{1\leqslant i\leqslant n}d_{X_i}(x_i,y_i)\leqslant
\varepsilon, \quad \mbox{i.e,}
$$
\begin{equation}\label{eq3.8}
\begin{split}
A\cap\bar{b}^{\nabla}_{\varepsilon}=\{d((x_1&,...,x_n),(y_1,...,y_n)):(x_1,...,x_n),(y_1,...,y_n)\in
P\\
& \mbox{ and } \rho_{\infty}((x_1,...,x_n),(y_1,...,y_n))\leqslant
\varepsilon\}.
\end{split}
\end{equation}
This equality, ~(\ref{eq3.4}) and ~(\ref{eq3.7}) imply
~(\ref{eq3.6}). Since $d$ is distance-increasing, condition (ii)
follows (see Remark~\ref{rem2.3}).

\textbf{(ii)}$\Rightarrow$\textbf{(i)}. Suppose that (ii) holds.
Lemma~ \ref{lem3.3} implies that there is an isotone function
$\Phi:D_{X_1}\times\cdots \times D_{X_n}\to\mathbb{R}_+$ such that
$$
d((x_1,...,x_n),(y_1,...,y_n))=\Phi(d_{X_1}(x_1,y_1),...,d_{X_n}(x_n,y_n))
$$
for all $(x_1,...,x_n)$, $(y_1,...,y_n)\in P$. To prove that $d$ has
an isotone continuation  we must find an isotone
$\Psi:\mathbb{R}^n_+\to\mathbb{R}_+$ for which
$$
\Psi |_{D_{X_1}\times \cdots \times D_{X_n}}=\Phi.
$$
To this end it suffices to prove the inequality
\begin{equation}\label{eq3.9}
\sup \{\Phi(\bar{x}):\bar{x}\in(D_{X_1}\times \cdots \times
D_{X_n})\cap \bar{b}^{\nabla}\}<\infty
\end{equation}
for every $\bar{b}\in\mathbb{R}^n_+$ (see Lemma~\ref{lem2.2}). It is
clear that $\bar{b}\leqslant\bar{b}_{\varepsilon}$ if $\varepsilon
=\max\limits_{1\leqslant i\leqslant n} b_i$. Hence $(D_{X_1}\times
\cdots \times D_{X_n})\cap \bar{b}^{\nabla}\subseteq (D_{X_1}\times
\cdots \times D_{X_n})\cap \bar{b}^{\nabla}_{\varepsilon}$.
Consequently inequality~(\ref{eq3.9}) follows from
$$
\sup\{\Phi(\bar{x}):\bar{x}\in (D_{X_1}\times \cdots \times
D_{X_n})\cap \bar{b}^{\nabla}_{\varepsilon}\}<\infty.
$$
Since (ii) holds, the last inequality is a consequence
of~(\ref{eq3.8}).
\end{proof}

Now we construct a distance-increasing metric which does not have
any isotone continuation. For the sake of simplicity, consider a
degenerate Cartesian product $P=X_1\times\cdots\times X_n$ with
$n=1$.
\begin{example}\label{ex3.6}
Let $X=[0,1)$ and let $\varphi:[0,1)\to[0,\infty)$ be a continuous
strictly increasing function with $\lim\limits_{t\to
1}\varphi(t)=\infty$ and $\lim\limits_{t\to\infty}\varphi(t)=0$.
Define a metric $\rho$ on $X$ by the rule
\begin{equation*}
\rho(x,y)=
\begin{cases}
\max\{x,y\} \quad\text{if } x\neq y \\
0 \quad\text{if } x=y.
\end{cases}
\end{equation*}
It is easy to see that $(X,\rho)$ is an ultrametric space, i.e., the
strong triangle inequality
$$
\rho(x,y)\leqslant \max\{\rho(x,z), \rho(z,y)\}
$$
holds for all $x, y ,z \in X$. The function $d:X\times
X\to\mathbb{R}_+$,
$$
d(x,y)=\varphi(\rho(x,y)), \quad x,y \in X
$$
is also a metric and even an ultrametric on $X$. Note that $d$ is
distance-increasing because $\varphi$ is increasing. We show that
the  identical mapping  $\id (x)=x$ is \textbf{uniformly continuous
but not bornologous} if we consider $\id$ as a mapping from
$(X,\rho)$ to $(X,d)$. The uniform continuity of $\id$ follows
directly from limit relation $\lim\limits_{t\to 0}\varphi(t)=0$. If
$\id$ is bornologous, then, by Remark~\ref{rem2.3}, there is an
increasing function $g:\mathbb{R}_+\to\mathbb{R}_+$ such that
$d(x,y)=\varphi(\rho(x,y))\leqslant g(\rho(x,y))$ for all $x,y \in
[0,1)$. Since $g$ is increasing and $\sup\limits_{x,y\in
X}\rho(x,y)\leqslant 1$, the inequality $d(x,y)\leqslant g(1)$ holds
for all $x,y \in X$. Hence the metric space $(X,d)$ is bounded,
contrary to the definition. By Theorem~\ref{th3.5}
$d=\varphi\circ\rho$ is a distance-increasing metric which does not
have isotone continuations.
\end{example}
 Analyzing the previous example we obtain the
following conditions under which the isotone metric products have
isotone continuations.
\begin{corollary}\label{cor3.7}
Let $(X_1,d_{X_1}),...,(X_n,d_{X_n})$ be metric spaces. Suppose that
each $(X_i,d_{X_i})$ is either unbounded ar there are $x_i,y_i \in
X_i$ such that
$$
d_{X_i}(x_i,y_i)=\diam X_i
$$
where $\diam X_i = \sup\{d_{X_i}(x,y):x,y\in X_i\}$. Then every
isotone metric product  on $P=X_1\times\cdots\times X_n$ has an
isotone continuation.
\end{corollary}
The next theorem gives us an intrinsic description of isotone metric
products having isotone amenable continuations.

\begin{theorem}\label{th3.8}
Let $(X_1,d_{X_1}),...,(X_n,d_{X_n})$ be metric spaces. The
following conditions are equivalent for every metric $d$ on
$P=X_1\times\cdots\times X_n$.
\begin{itemize}
  \item [(i)] $d$ is an isotone  metric product with an isotone amenable continuation.
  \item [(ii)] $d$ is distance-increasing and the identical mapping
  $$
   \id (x_1,...,x_n)=(x_1,...,x_n), \quad (x_1,...,x_n)\in P
  $$
  is bornologous as a mapping from $(P,\rho_{\infty})$ to $(P,d)$ and uniformly continuous
  as a mapping from $(P,d)$ to $(P,\rho_{\infty})$.
\end{itemize}
\end{theorem}
\begin{proof}
\textbf{(i)}$\Rightarrow$\textbf{(ii)} Suppose that condition (i) is
fulfilled. Then, in accordance with Theorem~\ref{th3.5}, $d$ is
distance-increasing and $\id$ is bornologous mapping from
$(P,\rho_{\infty})$ to $(P,d)$. Thus to prove (ii) it suffices to
show that $\id$ is an uniformly continuous mapping from $(P,d)$ to
$(P,\rho_{\infty})$. Suppose the contrary. Then there are
$\varepsilon>0$ and two sequences
$\{(x_1^i,...,x_n^i)\}_{i\in\mathbb{N}}$,
$\{(y_1^i,...,y_n^i)\}_{i\in\mathbb{N}}$ such that
\begin{equation}\label{eq3.10}
\lim\limits_{i\to\infty}d((x_1^i,...,x_n^i),(y_1^i,...,y_n^i))=0
\end{equation}
and
$$
\rho_{\infty}((x_1^i,...,x_n^i),(y_1^i,...,y_n^i))\geqslant
\varepsilon
$$
for every $i \in \mathbb{N}$. Passing to subsequences and reordering
the factors in $X_1\times\cdots\times X_n$ we may assume that
\begin{equation}\label{eq3.11}
d_{X_1}(x_1^i,y_1^i)\geqslant\varepsilon, \mbox{ for every } i \in
\mathbb{N}.
\end{equation}
By condition (i), there is an isotone amenable function
$\Phi:\mathbb{R}^n_+\to\mathbb{R}_+$ for which
$$
d((x_1,...,x_n),(y_1,...,y_n))=\Phi(d_{X_1}(x_1,y_1),...,d_{X_n}(x_n,y_n))
$$
for all $(x_1,...,x_n)$, $(y_1,...,y_n)\in P$. Using~(\ref{eq3.11})
we obtain
$$
d((x_1^i,...,x_n^i),(y_1^i,...,y_n^i))=\Phi(d_{X_1}(x_1^i,y_1^i),...,d_{X_n}(x_n^i,y_n^i))
$$
$$
\geqslant \Phi(d_{X_1}(x_1^i,y_1^i),0,...,0)\geqslant
\Phi(\varepsilon,0,...,0),
$$
contrary to~(\ref{eq3.10}). Consequently $\id$ is an uniformly
continuous mapping from $(P,d)$ to $(P,\rho_{\infty})$. Condition
(ii) follows.

\textbf{(i)}$\Rightarrow$\textbf{(ii)} Let condition (ii) hold. Then
by Theorem~\ref{th3.5} $d$ is an isotone metric product having an
isotone continuation. Let
$\Phi^{\circ}:\mathbb{R}^n_+\to\mathbb{R}_+$ be an isotone function
satisfying
$$
d((x_1,...,x_n),(y_1,...,y_n))=\Phi^{\circ}(d_{X_1}(x_1,y_1),...,d_{X_n}(x_n,y_n))
$$
for all $(x_1,...,x_n)$, $(y_1,...,y_n)\in P$. Define the function
$\Phi$ on the set \mbox{$A:=D_{X_1}\times\cdots\times D_{X_n}$} as
$$
\Phi:=\Phi^{\circ}|_{A}.
$$
It is clear that $\bar{0} \in A$ and $\Phi:A\to\mathbb{R}_+$ is
isotone and amenable. In accordance with Lemma~\ref{lem2.6},
condition (i) holds if
\begin{equation}\label{eq3.12}
\inf(\pr_j(B))=0,\quad  j=1,...,n,
\end{equation}
for every $B\subseteq A$ with $\inf(\Phi(B))=0$ and
$$
\sup\{\Phi(\bar{x}):\bar{x}\in A\cap\bar{b}^{\nabla}\}<\infty
$$
for  $\bar{b}\in\mathbb{R}^n_+$. The last inequality holds because
$\Phi(\bar{x})\leqslant \Phi^{\circ}(\bar{b})$ for every $\bar{x}\in
A\cap\bar{b}^{\nabla}$. Let $B\subseteq A$ and $\inf (\Phi(B))=0$.
The last equality implies that there are some sequences
$\{(x_1^i,...,x_n^i)\}_{i\in\mathbb{N}}$ and
$\{(y_1^i,...,y_n^i)\}_{i\in\mathbb{N}}$ such that
 $(d_{X_1}(x_1^i,y_1^i),...,d_{X_n}(x_n^i,y_n^i))\in B$ for all $i$
 and
\begin{equation}\label{eq3.13}
\lim\limits_{i\to\infty}d((x_1^i,...,x_n^i),(y_1^i,...,y_n^i))=0.
\end{equation}
Since the identical mapping $\id$ from $(P,d)$ to
$(P,\rho_{\infty})$ is uniformly continuous, equality~(\ref{eq3.13})
implies
$$
\lim\limits_{i\to\infty}\rho_{\infty}((x_1^i,...,x_n^i),(y_1^i,...,y_n^i))=0.
$$
The last equality and the inequalities
$$
d_{X_j}(x_j^i,y_j^i)\leqslant
\rho_{\infty}((x_1^i,...,x_n^i),(y_1^i,...,y_n^i)),\quad j=1,...,n,
$$
imply~(\ref{eq3.12}) for each $j\in\{1,...,n\}$.
\end{proof}

In the following theorem we show that every isotone subadditive
metric product has an isotone subadditive amenable continuation.
\begin{theorem}\label{th3.9}
Let $(X_1,d_{X_1}),...,(X_n,d_{X_n})$ be nonempty metric spaces. The
following conditions are equivalent for every metric $d$ on
$P=X_1\times\cdots\times X_n$.
\begin{itemize}
  \item [(i)] There is an isotone metric preserving function
  $\Psi:\mathbb{R}^n_+\to\mathbb{R}_+$ such that
  \begin{equation}\label{eq3.14}
    d((x_1,...,x_n),(y_1,...,y_n))=\Psi(d_{X_1}(x_1,y_1),...,d_{X_n}(x_n,y_n))
  \end{equation}
for all $(x_1,...,x_n), (y_1,...,y_n) \in P$.
  \item [(ii)] $d$ is an isotone subadditive metric product.
\end{itemize}
\end{theorem}
\begin{proof}
The implication (i)$\Rightarrow$(ii) is almost evident (see
Remark~\ref{rem2.8}).

Suppose condition (ii)  holds. Then there is an isotone subadditive
function $\Phi:A\to \mathbb{R}_+$ such that
$$
  d((x_1,...,x_n),(y_1,...,y_n))=\Phi(d_{X_1}(x_1,y_1),...,d_{X_n}(x_n,y_n))
$$
for all $(x_1,...,x_n), (y_1,...,y_n)\in P$ where
$A=D_{X_1}\times\cdots \times D_{X_n}$. Let
$$
J_0:=\{j\in \{1,...,n\}:\card X_j=1\}.
$$
It is clear that $j\in J_0$ if and only if $\pr_j(\bar{a})=0$ for
every $\bar{a}\in A$. Let $\bar{e}^1=(1,0,...,0)$,
$\bar{e}^2=(0,1,...,0)$,...,$\bar{e}^n=(0,0,...,1)$ be the standard
basis vectors in $\mathbb{R}^n$. Denote by $B$ the Minkowski sum of
$A$ and of the convex cone generated by the set $\{\bar{e}^j:j\in
J_0\}$, i.e., $\bar{x}\in B$ if and only if there are
$\alpha_j\geqslant 0$, $j\in J_0$, and $\bar{a} \in A$ such that
 \begin{equation}\label{eq3.15}
\bar{x}=\bar{a}+\sum\limits_{j\in J_0}\alpha_j \bar{e}^j
  \end{equation}
  It is clear that $A\subseteq B$.
Since the linear space $\mathbb{R}^n$ is a direct sum of linear
spaces generating by the vectors $\bar{e}^j$, $j\in J_0$, and by the
set $A$, the equality
$$
\bar{a}+\sum\limits_{j\in J_0}\alpha_j \bar{e}^j =
\bar{b}+\sum\limits_{j\in J_0}\beta_j \bar{e}^j
$$
with $\bar{a}$, $\bar{b}\in A$ implies $\bar{a}=\bar{b}$ and
$\sum\limits_{j\in J_0}\alpha_j \bar{e}^j = \sum\limits_{j\in
J_0}\beta_j \bar{e}^j$. Let $c$ be a strictly positive constant.
Define a function  $\Phi^{\circ}:B\to\mathbb{R}_+$ by the rule: if
$\bar{x}\in B$ has a representation~(\ref{eq3.15}), then
\begin{equation}\label{eq3.16}
\Phi^{\circ}(\bar{x})=\left\{
            \begin{array}{ll}
              \Phi(\bar{a}) & \hbox{if  } \sum\limits_{j\in J_0}\alpha_j \bar{e}^j=0; \\
               \Phi(\bar{a}) +c & \hbox{if  } \sum\limits_{j\in J_0}\alpha_j \bar{e}^j\neq 0.
            \end{array}
          \right.
\end{equation}
The uniqueness of representation~(\ref{eq3.15}) implies that
$\Phi^{\circ}$ is correctly  defined and, moreover,
from~(\ref{eq3.16}) it follows that $\Phi^{\circ}|_{A}=\Phi$.  As in
the second part of the proof of Lemma~\ref{lem2.9} we can show that
$\Phi^{\circ}$ is isotone and subadditive. Moreover it is easy to
see that the set $\textbf{S}(\bar{x})=\textbf{S}(\bar{x},B)$,
defined as in the proof of Lemma~\ref{lem2.9}, is nonempty.
Consequently, the function $\Psi:\mathbb{R}^n_+\to\mathbb{R}_+$,
$$
\Psi(\bar{x})=\inf\{\sum\limits_{i=1}^m\bar{x}^i:(\bar{x}^1,...,\bar{x}^m)\in
\bar{\textbf{S}}(\bar{x},B)\},
$$
(cf.~(\ref{eq2.20})) is an isotone subadditive continuation of
$\Phi^{\circ}$. Thus~(\ref{eq3.14}) holds for all
$(x_1,...,x_n),(y_1,...,y_n)\in P$. Now to prove $(i)$ it is
sufficient  to show that $\Psi$ is amenable. Let us do it.

Equality~(\ref{eq3.16}) and $\Psi|_{A}=\Phi^{\circ}|_{A}=\Phi$ imply
that $\Psi(\bar{0})=0$. Let $\bar{x}=(x_1,...,x_n)$ be a point of
$\mathbb{R}^n_+$ such that $\Psi(\bar{x})=0$. We shall show that
$x_j=\pr_j(\bar{x})=0$ for every $j=1,...,n$. Suppose there is
$j_0\in J_0$ such that $x_{j_0}>0$. Since $\bar{x}\geqslant
\pr_{j_0}(\bar{x})\bar{e}^{j_0}$ and $\Psi$ is isotone, we have
$$
\Psi(\bar{x})\geqslant
\Psi(\pr_{j_0}(\bar{x})\bar{e}^{j_0})=\Phi^{\circ}(\pr_{j_0}(\bar{x})\bar{e}^{j_0})=c>0.
$$
Hence if $\pr_{j_0}(\bar{x})>0$, then $j_0\in \{1,...,n\}\backslash
J_0$. The membership $j_0\in\{1,...,n\}\backslash J_0$ and the
definition of $J_0$ imply that $\card X_{j_0}\geqslant 2$.
Consequently there are $z_{j_0}$, $y_{j_0}\in X_{j_0}$ such that
$d_{X_{j_0}}(z_{j_0},y_{j_0})>0$. Hence there is $m \in \mathbb{N}$
for which $d_{X_{j_0}}(z_{j_0},y_{j_0})\leqslant
m\pr_{j_0}(\bar{x})$ so that
\begin{equation}\label{eq3.17}
    d_{X_{j_0}}(z_{j_0},y_{j_0})\bar{e}^{j_0}\leqslant
    m\pr_{j_0}(\bar{x})\bar{e}^{j_0}.
\end{equation}
It is easy to see that $ d_{X_{j_0}}(z_{j_0},y_{j_0})\bar{e}^{j_0}
\in A$. Since $\Psi$ is isotone and subadditive,
inequality~(\ref{eq3.17}) implies
\begin{equation*}
\begin{split}
\Phi( d_{X_{j_0}}(z_{j_0},y_{j_0})&\bar{e}^{j_0})=\Psi(
d_{X_{j_0}}(z_{j_0},y_{j_0})\bar{e}^{j_0})\leqslant \Psi(m
\pr_{j_0}(\bar{x})\bar{e}^{j_0}) \\ &\leqslant m
\Psi(\pr_{j_0}(\bar{x})\bar{e}^{j_0})\leqslant m \Psi(\bar{x})=0.
\end{split}
\end{equation*}
Hence $\Phi$ is not amenable, contrary to the definition. Thus if
$\Psi(\bar{x})=0$ and $\pr_j(\bar{x})>0$, then $j\notin J_0$ and
$j\notin \{1,...,n\}\backslash J_0$, i.e., the implication
$\Psi(\bar{x}=0)\Rightarrow (\bar{x}=0)$ holds. The function $\Psi$
is amenable, as required.
\end{proof}
Theorems ~\ref{th3.9} and  ~\ref{th3.8} imply the following
\begin{corollary}\label{cor3.10}
Let $(X_1,d_{X_1}),...,(X_n,d_{X_n})$ be nonempty metric spaces. If
$d$ is an isotone subadditive metric product on $P$, then the
identical mapping
$$
\id (x_1,...,x_n) =(x_1,...,x_n), \quad (x_1,...,x_n) \in P,
$$
is bornologous as a mapping from $(P,\rho_{\infty})$ to $(P,d)$ and
uniformly continuous as a mapping from $(P,d)$ to
$(P,\rho_{\infty})$.
\end{corollary}

\section{Modulus of continuity of bornologous \mbox{functions}}
We shall say that a function $f:\mathbb{R}_+^n\to \mathbb{R}_+$ is
nonconstant with respect to the variable $x_1$ if there exist
nonnegative numbers $a_i, i=2,...,n$, such that the function
$f(x_1,a_2,\dots,a_n)$ is nonconstant as a function of one variable
$x_1$. Analogously we define the functions
$f:\mathbb{R}^n_+\to\mathbb{R}_+$ which are nonconstant w.r.t. the
variable $x_i$ for $2\leqslant i \leqslant n$.

Let us denote by $\mathcal{W}$ the family all bornologous functions
$g:\mathbb{R}^n_+\to\mathbb{R}_+$ which are nonconstant w.r.t. every
variable.

It is easy to prove that every uniformly continuous function
$g:\mathbb{R}_+^n\to \mathbb{R}_+$ which is nonconstant w.r.t. all
variables belongs to $\mathcal{W}$ (cf. Example~\ref{ex3.6}).

For $g\in \mathcal{W}$ we define the ``\emph{modulus of
continuity}'' as follows
\begin{equation}\label{eq4.1}
\omega(g,\overline{\varepsilon}):=\sup\limits_{\substack{|\overline{x}-\overline{y}|\leqslant\overline{\varepsilon}
\\  \overline{x},\overline{y}\in
\mathbb{R}_+^n}}|g(\overline{x})-g(\overline{y})|
\end{equation}
where
$\bar{\varepsilon}=(\varepsilon_1,...,\varepsilon_n)\in\mathbb{R}^n_+$
and the inequality $|\bar{x}-\bar{y}|\leqslant\bar{\varepsilon}$ is
understood in accordance with~(\ref{eq2.1}).
\begin{remark}\label{rem4.1}
For $\bar{\varepsilon}=\bar{0}$ formula~(\ref{eq4.1}) gives us the
equality $\omega(\bar{g},\bar{0})=0$.
\end{remark}

It was noted in~\cite{BD2} that the inequality
$$
    |f(a)-f(b)|\leqslant f(|a-b|)
$$
holds for every metric preserving function
$f:\mathbb{R}_+\to\mathbb{R}_+$ and all $a,b \in \mathbb{R}_+$. In
the following lemma we give a multivariate version of this
inequality for isotone metric preserving functions of several
variables.
\begin{lemma}\label{lem4.2}
Let $F\in \mathfrak{F}^n_i$. The inequality
\begin{equation}\label{eq4.2}
|F(\bar{x})-F(\bar{y}))|\leqslant F(|\bar{x}-\bar{y}|)
\end{equation}
holds for all $\bar{x}, \bar{y} \in \mathbb{R}^n_+$ with
$|\bar{x}-\bar{y}|=(|x_1-y_1|,...,|x_n-y_n)|)$.
\end{lemma}
\begin{proof}
Let $\bar{x}, \bar{y} \in \mathbb{R}^n_+$. We may assume, without
loss of generality, that
\begin{equation}\label{eq4.3}
F(\bar{x})\geqslant F(\bar{y}).
\end{equation}
It is easy to see that
\begin{equation}\label{eq4.4}
\bar{x}\leqslant \bar{y}+|\bar{x}-\bar{y}|.
\end{equation}
Since $F$ is isotone and subadditive (see Theorem~\ref{th3.1}),
inequalities~(\ref{eq4.3}) and ~(\ref{eq4.4}) imply  ~(\ref{eq4.2}),
$$
|F(\bar{x})-F(\bar{y})|=F(\bar{x})-F(\bar{y})\leqslant
F(\bar{y}+|\bar{x}-\bar{y}|)-F(\bar{y})
$$
$$
\leqslant
F(\bar{y})+F(|\bar{x}-\bar{y}|)-F(\bar{y})=F(|\bar{x}-\bar{y}|).
$$
\end{proof}
\begin{theorem}\label{th4.3}
Let $F$ be an arbitrary nonnegative function defined on
$\mathbb{R}_+^n$. Then $F$ belongs to $\mathfrak{F}_i^n$ if and only
if there exists $g\in \mathcal{W}$ such that
\begin{equation}\label{eq4.5}
F(\overline{\varepsilon})=\omega(g,\overline{\varepsilon})
\end{equation}
for every $\overline{\varepsilon} \in \mathbb{R}_+^n$.
\end{theorem}
\begin{proof}
Suppose $F\in \mathfrak{F}_i^n$. Let us prove that there exists
$g\in \mathcal{W}$ such that equality~(\ref{eq4.5}) holds for every
$\overline{\varepsilon} \in \mathbb{R}_+^n$.
 It suffices to show that
\begin{equation}\label{eq4.6}
F \in \mathcal{W}
\end{equation}
 and
\begin{equation}\label{eq4.7}
F(\overline{\varepsilon})=\omega(F,\overline{\varepsilon}).
\end{equation}
 Let us verify that $F$ is nonconstant w.r.t. each variable. Indeed, since $F$ is amenable we have
 $$
 0=F(0,...,0)\neq F(t,0,...,0)\mbox{ if } t>0.
 $$
 Thus, $F$ is nonconstant w.r.t. the first variable.
 Similarly we can establish that $F$ is nonconstant w.r.t. other variables.
 Thus, to prove~(\ref{eq4.6}) it remains to verify that for every
 $\varepsilon>0$ there exists $\delta(\varepsilon)>0$ such that
 for every $\overline{x}=(x_1,...,x_n)$, $\overline{y}=(y_1,...,y_n)\in \mathbb{R}_+^n$
 the inequality
\begin{equation}\label{eq4.8}
\sum\limits_{i=1}^n|x_i-y_i|<\varepsilon
\end{equation}
implies the inequality
\begin{equation}\label{eq4.9}
 |F(\overline{x})-F(\overline{y})|\leqslant \delta(\varepsilon).
\end{equation}
Suppose that inequality~(\ref{eq4.8}) holds. By Lemma~\ref{lem4.2}
we obtain
$$
|F(\bar{x})-F(\bar{y})|\leqslant F(|\bar{x}-\bar{y}|).
$$
Inequality~(\ref{eq4.8}) implies that
$|\bar{x}-\bar{y}|\leqslant\bar{\varepsilon}$ where
$\bar{\varepsilon}=(\varepsilon,...,\varepsilon) \in
\mathbb{R}^n_+$. Consequently we obtain
$F(|\bar{x}-\bar{y}|)\leqslant F(\bar{\varepsilon})$. Thus
inequality~(\ref{eq4.9}) holds with
$\delta(\varepsilon)=F(\bar{\varepsilon})$. It still remains to note
that equality~(\ref{eq4.7}) follows directly from the definition of
modulus of continuity and Lemma~\ref{lem4.2}.

Suppose $g \in \mathcal{W}$ and  equality~(\ref{eq4.5}) holds for
every $\bar{\varepsilon} \in \mathbb{R}_+^n$. It is necessary to
show that $F\in \mathfrak{F}_i^n$. Note that monotonicity and
nonnegativity $F$ follow directly from~(\ref{eq4.1}). Let us show
that $F(\overline{\varepsilon})>0$ if $\overline{\varepsilon}>
\overline{0}$. By Remark~\ref{rem4.1} we have $F(\bar{0})=0$

The inequality $\bar{\varepsilon}>\bar{0}$ implies that there exists
$i_0\in \{1,...,n\}$ such that $\varepsilon_{i_0}>0$. For
convenience we may consider the case when $i_{0}=1$.
In~(\ref{eq4.5}) the function $g$ belongs to $\mathcal{W}$.
Consequently $g$ is nonconstant w.r.t. all variables. Hence there
are some numbers $a_i\in \mathbb{R}_+$, $i=2,...,n$, such that the
function
$$
g^*(t):=g(t,a_2,...,a_n), \quad t \in \mathbb{R}_+,
$$
is nonconstant. Suppose $F(\varepsilon)=0$. Then we obtain
$g^*(t_1)=g^*(t_2)$ whenever $|t_2-t_1|\leqslant \varepsilon_1$.
Consequently $g^*(t) = \mathrm{constant}$. This contradiction
implies the desired inequality $F(\bar{\varepsilon})>0$ for
$\bar{\varepsilon}>\bar{0}$.

It remains to verify that
\begin{equation}\label{eq4.10}
F(\overline{x}+\overline{y})\leqslant
F(\overline{x})+F(\overline{y})
\end{equation}
for $\bar{x},\bar{y}\in \mathbb{R}_n^+$. To prove~(\ref{eq4.10})
consider  $\bar{u},\bar{v}\in\mathbb{R}_+^n$ for which
$|\overline{u}-\overline{v}|\leqslant \overline{x}+\overline{y}$.
The simple geometrical reasoning shows that there exists
$\overline{w}\in \mathbb{R}_+^n$ such that
$$
|\overline{u}-\overline{w}|\leqslant \overline{x} \mbox{ and }
|\overline{w}-\overline{v}|\leqslant \overline{y}.
$$
Now using~(\ref{eq4.1}) and the two inequalities given above we get
$$
|g(\overline{u})-g(\overline{v})|\leqslant
|g(\overline{u})-g(\overline{w})|+|g(\overline{w})-g(\overline{v})|\leqslant
\omega(g,\overline{x})+\omega(g,\overline{y})=F(\overline{x})+F(\overline{y}).
$$
From the other hand
\begin{equation*}\
F(\overline{x}+\overline{y})=\omega(g,\overline{x}+\overline{y})=\sup\limits_{\substack{|\overline{u}-\overline{v}|\leqslant\overline{x}+\overline{y}
\\  \overline{x},\overline{y}\in
\mathbb{R}_+^n}}|g(\overline{u})-g(\overline{v})|,
\end{equation*}
which proves~(\ref{eq4.10}). Now the membership relation $F\in
\mathfrak{F}^n_i$ follows from Theorem~\ref{th3.1}.
\end{proof}
The first part of the proof of Theorem~\ref{th4.3} shows that
$\mathfrak{F}^n_i\subseteq \mathcal{W}$. The following selects the
metric preserving functions from the functions belonging to
$\mathcal{W}$.
\begin{corollary}\label{cor4.4}
The set $\mathfrak{F}^n_i$ coincides with the set of fixed points of
the mapping
$$
\mathcal{W}\ni g \mapsto \omega(g,\cdot) \in \mathcal{W}
$$
where $\omega(g,\cdot)$ is defined by~(\ref{eq4.1}).
\end{corollary}
\section{Some $3l$-universal subspaces of $\mathbb{R}$}

We shall say that a metric space $(X,d)$ is $3l$-universal (linear
three universal) if and only if every three-point metric subspace of
the real line is isometrically embeddable in $(X,d)$, i.e., for all
$a,b \in \mathbb{R}_+$ there are $x_1, x_2, x_3 \in X$ such that
\begin{equation}\label{eq5.1}
d(x_1,x_2) = a, \quad d(x_2,x_3) = b,\quad d(x_1,x_3)=a+b.
\end{equation}

The following theorem gives a set of sufficient and necessary
conditions under which an isotone function
$\Phi:\mathbb{R}^n_+\to\mathbb{R}_+$ belongs to $\mathfrak{F}^n_i$.
\begin{theorem}\label{th5.1}
Let $\Phi:\mathbb{R}^n_+\to\mathbb{R}_+$ be an isotone function. The
following conditions are equivalent.
\begin{itemize}
  \item [(i)] The function
  $\Phi(d_{X_1}(\cdot,\cdot),...,d_{X_n}(\cdot,\cdot))$ is a metric
  on $P=X_1\times\cdots\times X_n$ for all metric spaces
  $(X_1,d_{X_1}),...,(X_n,d_{X_n})$.
  \item [(ii)] The function   $\Phi(d_{X_1}(\cdot,\cdot),...,d_{X_n}(\cdot,\cdot))$
  is a metric on $P=\mathbb{R}^n$ with \mbox{$X_1=\cdots=X_n=\mathbb{R}$},
  $$
d_{X_1}(x_1,y_1)=|x_1-y_1|,...,d_{X_n}(x_n,y_n)=|x_n-y_n|.
  $$
  \item [(iii)] The function
  $\Phi(d_{X_1}(\cdot,\cdot),...,d_{X_n}(\cdot,\cdot))$ is a metric
  on \mbox{$P=X_1\times\cdots\times X_n$} for all three-point subspaces
  $X_1,...,X_n$ of $\mathbb{R}$ with the usual metric
  $d(x,y)=|x-y|$.
  \item [(iv)] There are $3l$-universal metric spaces
  $(X_1,d_{X_1}),...,(X_n,d_{X_n})$ such that
  $\Phi(d_{X_1}(\cdot,\cdot),...,d_{X_n}(\cdot,\cdot))$ is a metric
  on $P=X_1\times\cdots\times X_n$.
\end{itemize}
\end{theorem}
\begin{proof}[\textbf{Sketch of the proof.}] Condition (i) is a
reformulation of the definition of the set $\mathfrak{F}^n_i$. The
implications (i)$\Rightarrow$(ii), (i)$\Rightarrow$(iii),
(i)$\Rightarrow$(iv) are evident. Each condition from (ii), (iii),
(iv) implies that for every $(a_1,...,a_n)$, $(b_1,...,b_n)\in
\mathbb{R}^n_+$ there are triangles $\{x_i,y_i,z_i\}\subseteq X_i$
such that
$$
d_{X_i}(x_i,y_i)=a_i, \quad d_{X_i}(y_i,z_i)=b_i, \quad
d_{X_i}(x_i,z_i)=a_i+b_i, \quad i=1,...,n.
$$
Since $\Phi(d_{X_1}(\cdot,\cdot),...,d_{X_n}(\cdot,\cdot))$ is a
metric  on $P$, the triangle inequality implies
$$
\Phi(a_1+b_1,...,a_n+b_n) \leqslant
\Phi(a_1,...,a_n)+\Phi(b_1,...,b_n).
$$
Hence $\Phi$ is subadditive. The implications (ii)$\Rightarrow$(i),
(iii)$\Rightarrow$(i), (iv)$\Rightarrow$(i) follow from
Theorem~\ref{th3.1}.
\end{proof}

\begin{remark}\label{rem5.2} Theorem~\ref{th5.1} has a natural
analog for arbitrary metric preserving functions
$\Phi:\mathbb{R}^n_+\to\mathbb{R}_+$. For example, Theorem 2.1 in
~\cite{HM} or Lemma 1 in~\cite{BFS} or Lemma 7 in Section 5 of
~\cite{FS} are similar to the equivalence (i)$\Leftrightarrow$(iii)
from Theorem~\ref{th5.1}.
\end{remark}

Let us consider some  $3l$-universal metric subspaces of
$\mathbb{R}=(-\infty,\infty)$.

The set $\mathbb{R}$ with the usual addition and the multiplication
on rational numbers forms a vector space over the field
$\mathbb{Q}$. If $X$ is a linear subspace of $\mathbb{R}$ and
$B\subseteq X$, then $B$ is called a \emph{Hamel basis} of $X$ if
and only if every nonzero $x\in X$ can be uniquely  represented as a
finite linear combination
$$
x=\sum\limits_{i=1}^nr_ib_i
$$
where $b_1,...,b_n$ are distinct elements of $B$ and
$r_1,...,r_n\in\mathbb{Q}\backslash\{0\}$. It is well known and easy
to prove by transfinite induction that such basis exists for every
linear subspace of $\mathbb{R}$. In particular $\varnothing$ is a
Hamel basis for $X=\{0\}$.
\begin{proposition}\label{prop5.3}
Let $X$ be a linear subspace of $\mathbb{R}$. If
$\mathbb{R}\backslash X\neq\varnothing$, then the set
$\mathbb{R}\backslash X$ is $3l$-universal.
\end{proposition}
\begin{proof}
Let $X\neq\mathbb{R}$. It is sufficient to show that for every
$a,b\in\mathbb{R}_+$ there are $x_0,y_0,z_0\in\mathbb{R}\backslash
X$ such that
\begin{equation}\label{eq5.2}
x_0-y_0 = a \mbox{ and } y_0-z_0 = b.
\end{equation}
This is evident for $X=\{0\}$. Suppose that $X\neq\{0\}$ and that
$B$ is a Hamel basis of $X$. Let $x^*$, $y^*$ and $z^*$ be some
nonzero points of $\mathbb{R}$ such that
\begin{equation}\label{eq5.3}
x^*-y^* = a \mbox{ and } y^*-z^* = b.
\end{equation}
Since $B$ is a rationally independent set, there is a Hamel basis
$B_1$ of $\mathbb{R}$ such that $B_1\supseteq B$. The numbers $x^*,
y^*, z^*$ can be uniquely represented as
\begin{equation}\label{eq5.4}
x^*=\sum\limits_{i=1}^{n(x)}r_i(x)b_i(x),\quad
y^*=\sum\limits_{i=1}^{n(y)}r_i(y)b_i(y) \mbox{ and }
z^*=\sum\limits_{i=1}^{n(z)}r_i(z)b_i(z)
\end{equation}
with $r_i(x), r_i(y), r_i(z) \in \mathbb{Q}\backslash \{0\}$ and
$b_i(x), b_i(y), b_i(z) \in B_1$. Since $$\mathbb{R}\supseteq X\neq
\mathbb{R},$$ the set $B_1\backslash B$ is nonempty. Let $b_0\in
B_1\backslash B$. Define $x_0, y_0$ and $z_0$ as
\begin{equation}\label{eq5.5}
x_0 = x^*+(r_0+1)b_0,\quad y_0 = y^*+(r_0+1)b_0, \quad z_0 =
z^*+(r_0+1)b_0
\end{equation}
where
\begin{equation}\label{eq5.6}
r_0:=\max \left\{ \bigvee_{i=1}^{n(x)}|r_i(x)|,
\bigvee_{i=1}^{n(y)}|r_i(y)|, \bigvee_{i=1}^{n(z)}|r_i(z)| \right\}
\end{equation}
The numbers $x_0$, $y_0$, $z_0$ can be represented as finite
rational nonzero linear combinations of elements of $B_1$. The
uniqueness of such representations and equalities~(\ref{eq5.4}),
~(\ref{eq5.5}), ~(\ref{eq5.6}) imply that $x_0, y_0, z_0 \in
\mathbb{R}\backslash X$. Moreover~(\ref{eq5.2}) follows from
~(\ref{eq5.3})
 and ~(\ref{eq5.5}).
\end{proof}
\begin{corollary}\label{cor5.4}
Let $X$ be a subfield of the field $\mathbb{R}$. If
$\mathbb{R}\backslash X\neq \varnothing$, then $\mathbb{R}\backslash
X$ is $3l$-universal.
\end{corollary}
\begin{proof}
Every field $X\subseteq \mathbb{R}$ can be regarded as a linear
space over the field $\mathbb{Q}$ of rational numbers.
\end{proof}

\begin{corollary}\label{cor5.5}
The set $T$ of all real transcendental numbers is $3l$-universal
\end{corollary}
\begin{proof}
$\mathbb{R}\backslash T$ is the set of all real algebraic numbers.
It is well known that this is a subfield of the field $\mathbb{R}$.
\end{proof}

\begin{proposition}\label{prop5.6}
Let $X$ be a linear subspace of $\mathbb{R}$. Then $X$ is
$3l$-universal if and only if $X=\mathbb{R}$.
\end{proposition}
\begin{proof}
It is clear that $\mathbb{R}$ is $3l$-universal. Suppose that $X\neq
\mathbb{R}$. Let $a$ be a positive number from $\mathbb{R}\backslash
X$. If $X$ is $3l$-universal, then there is $x,y \in X$ such that
$|x-y| = a$. We may assume that $x\geqslant y$. Hence we obtain the
contradiction $X\ni (x-y) = |x-y|\notin X$.
\end{proof}
A simple modification of the proof of Proposition~\ref{prop5.3}
shows that the implication
\begin{itemize}
  \item if $\mathbb{R}\backslash X \neq \varnothing$, then there is
  an isometric embedding $f:A\to\mathbb{R}\backslash X$
\end{itemize}
holds for every finite subset $A$ of $\mathbb{R}$ and every linear
subspace $X$ of $\mathbb{R}$. In particular, we obtain the next
interesting generalization of Corollary~\ref{cor5.5}.
\begin{corollary}\label{cor5.7}
Let $A$ be a finite subset  of $\mathbb{R}$. Then there is an
isometric embedding of $A$ in the set $T$ of all real transcendental
numbers.
\end{corollary}
In the following example we construct a set $X\subseteq
\mathbb{R}_+$ such that the distance set $D_X=\{|x-y|:x,y\in X\}$ is
the same as $\mathbb{R}_+$ but $X$ is not $3l$-universal. To this
end, we use the triadic Cantor set. Recall the definition. Let $x\in
[0,1]$ and expand $x$ as
\begin{equation}\label{eq5.7}
    x=\sum\limits_{n=1}^{\infty}\frac{b_n(x)}{3^n}, \quad
    b_n(x)\in\{0,1,2\}.
\end{equation}
The Cantor set $C$ is the set of points from $[0,1]$ which have
expansion~(\ref{eq5.7}) using only the digits $0$ and $2$. Thus
$x\in C$ if and only if $x$ has a triadic representation
\begin{equation}\label{eq5.8}
    x=\sum\limits_{n=1}^{\infty}\frac{2\alpha_m}{3^m}
\end{equation}
where $\alpha_m=\alpha_m(x) \in \{0,1\}$.
\begin{example}\label{ex5.8}
Define a set $C^e$ as
$$
C^e=\bigcup\limits_{n=0}^{\infty}3^nC
$$
where $3^nC=\{3^nx:x\in C\}$. It follows from~(\ref{eq5.8}) that a
real number $t$ belongs to $C^e$ if and only if $t$ has a base 3
expansion with the digits $0$ and $2$ only, i.e.,
\begin{equation}\label{eq5.9}
t=\sum\limits_{j=-\infty}^{M(t)}a_j(t)3^j
\end{equation}
where $M(t)\in \mathbb{Z}$ and $a_j(t)\in\{0,2\}$. It is easy to
prove (see, for example, \cite[Chapter 8, Example 3]{GO}) that the
distance set of the triadic Cantor set $C$ is $[0,1]$, $D_C=[0,1]$.
Consequently we obtain $D_{3^nC}=[0,3^n]$ for every natural number
$n$. Hence we have $\mathbb{R}_+\supseteq D_{C^e}\supseteq
\bigcup\limits_{n=0}^{\infty}[0,3^n]=\mathbb{R}_+$, so that
$D_{C^e}=\mathbb{R}_+$.
\end{example}
We claim that there are no points $x, y, z\in C^e$ such that
\begin{equation}\label{eq5.10}
|x-y|=\frac{1}{3}\mbox{ and } |x-z|=\frac{1}{6}
\end{equation}
Suppose contrary that $x, y, z\in C^e$ and satisfy~(\ref{eq5.10}).
Let $j_0$ be the largest index $j$ such that $a_j(x)\neq a_j(y)$
where $a_j(x)$ and $a_j(y)$ are the coefficients from
expansion~(\ref{eq5.9}) for $x$ and, respectively, $y$. Then we
obtain
$$
\frac{1}{3}=|x-y|\geqslant 2\cdot
3^{j_0}-|\sum\limits_{j=-\infty}^{j_0-1}(a_j(x)-a_j(y))3^j|\geqslant
2\cdot3^{j_0}-\sum\limits_{j=-\infty}^{j_0-1}|a_j(x)-a_j(y)|3^j
$$
$$
\geqslant
2\cdot3^{j_0}-\sum\limits_{j=-\infty}^{j_0-1}2\cdot3^j=2\cdot3^{j_0}-3\cdot3^{j_0-1}=3^{j_0}.
$$
Consequently the inequality $j_0\leqslant -1$ holds. Hence the
equality
\begin{equation}\label{eq5.11}
a_j(x)=a_j(y)
\end{equation}
holds for all $j\geqslant 0$. Similarly we have
\begin{equation}\label{eq5.12}
a_j(x)=a_j(z)
\end{equation}
for all $j\geqslant 0$. Equations~(\ref{eq5.11}) and ~(\ref{eq5.12})
imply that there is a constant $c_1\geqslant 0$ such that
$$
c_1=\sum\limits_{j=0}^{M(x)}a_j(x)3^j=\sum\limits_{j=0}^{M(y)}a_j(y)3^j=\sum\limits_{j=0}^{M(z)}a_j(z)3^j.
$$
Define the points $x^*, y^*, z^*$ as $x^*=x-c_1$, $y^*=y-c_1$ and
$z^*=z-c_1$. Then the equalities $|x^*-y^*|=\frac{1}{3}$ and
$|x^*-z^*|=\frac{1}{6}$ hold and $x^*, y^*, z^* \in C$. Since
$(\frac{1}{3}, \frac{2}{3})\cap C=\varnothing$ and
$C\subseteq[0,1]$, the equality $|x^*-y^*|=\frac{1}{3}$ shows that
the following six combinations
$$
(i_1)\quad x^*=0, \: y^*=\frac{1}{3},\qquad (i_2)\quad
x^*=\frac{1}{3}, \: y^*=\frac{2}{3},\qquad (i_3)\quad
x^*=\frac{2}{3},\:
 y^*=1,
$$
$$
(i_4)\quad x^*=\frac{1}{3},\:  y^*=0,\qquad (i_5)\quad
x^*=\frac{2}{3},\: y^*=\frac{1}{3},\qquad (i_6)\quad x^*=1,\:
 y^*=\frac{2}{3}
$$
are the only possible. This equality and the next one
$|x^*-z^*|=\frac{1}{6}$ give us the membership relation
$$
z^*\in \left\{-\frac{1}{6},\: \frac{1}{6},\: \frac{1}{2},
\:\frac{5}{6},\: \frac{7}{6}\right\}.
$$
Since $C\cap \left\{-\frac{1}{6},\: \frac{1}{6},\: \frac{1}{2},
\:\frac{5}{6},\: \frac{7}{6}\right\}=\varnothing$ we obtain that
$z^*\notin C$, contrary to the supposition $x^*, y^*, z^* \in C$.
Thus $C^e$ is not $3l$-universal.
\begin{remark}\label{rem5.9}
$C^e$ is the smallest (w.r.t. the relation $\subseteq$) set
$X\subseteq\mathbb{R}_+$ satisfying the conditions $C\subseteq X$
and $3X=X$.
\end{remark}

{\bf Oleksiy Dovgoshey}

Institute of Applied Mathematics and Mechanics of NASU, R. Luxemburg Str. 74, Donetsk 83114, Ukraine

{\bf E-mail: } aleksdov@mail.ru
\bigskip

{\bf Evgeniy Petrov}

Institute of Applied Mathematics and Mechanics of NASU, R. Luxemburg
Str. 74, Donetsk 83114, Ukraine

{\bf E-mail: } eugeniy.petrov@gmail.com
\bigskip

{\bf Galina Kozub}

Donetsk National University, Universitetska Str. 24, Donetsk 83055,
Ukraine

{\bf E-mail: } kozub\textunderscore galina@mail.ru

\end{document}